\documentclass[a4paper,10pt]{amsart}
\usepackage{amsmath,amsfonts,amssymb,amsthm, amsfonts}
\usepackage[latin1]{inputenc}
\usepackage[english]{babel}

\usepackage{mathptmx}

\usepackage[cmtip,arrow]{xy}
\usepackage{pb-diagram,pb-xy}

\swapnumbers 
\theoremstyle{plain}
\newtheorem{thm}{Theorem}[section]
\newtheorem{prop}[thm]{Proposition}
\newtheorem{lemma}[thm]{Lemma}

\theoremstyle{remark}
\theoremstyle{definition}
\newtheorem{rem}[thm]{Remark}
\newtheorem{rems}[thm]{Remarks}
\newtheorem{remdef}[thm]{Remark-Definition}
\newtheorem{remsdefs}[thm]{Remarks-Definitions}

\newtheorem{defi}[thm]{Definition}

\newtheorem{notas}[thm]{Notations}

\newtheorem{example}[thm]{Example}
\newtheorem{nota}[thm]{Notation}

\title[Elliptic isometries]{Elliptic isometries of the manifold of definite positive real matrices with the trace metric}
\author[A. Dolcetti]{Alberto Dolcetti}
\author[D. Pertici]{Donato Pertici}

\begin{document}

\parindent 0pt
\selectlanguage{english}

\maketitle

\vspace*{-0.2in}

\begin{center}
{\scriptsize Dipartimento di Matematica  e Informatica, Viale Morgagni 67/a, 50134 Firenze, ITALIA

\vspace*{0.07in}

alberto.dolcetti@unifi.it, \  http://orcid.org/0000-0001-9791-8122

\vspace*{-0.03in}

donato.pertici@unifi.it,  \   http://orcid.org/0000-0003-4667-9568}

\end{center}

\begin{abstract}
We study the differential-geometric properties of the loci of fixed points of the elliptic isometries of the manifold of definite positive real matrices with the trace metric. We also give an explicit description of such loci and in particular we find their De Rham decomposition.
\end{abstract}

{\small \tableofcontents}

{\small {\scshape{Keywords.}} Positive definite matrices, trace metric, Hadamard manifolds, (elliptic) isometries, (irreducible) symmetric Riemannian spaces, De Rham decomposition.}

{\small {\scshape{Mathematics~Subject~Classification~(2010):}}  15B48, 53C35.}

{\small {\scshape{Funding:}} this research has been partially supported by GNSAGA-INdAM (Italy).}

\section*{Introduction}
The Riemannian manifold $(\mathcal{P}_n, g)$ of symmetric positive definite real matrices endowed with the trace metric has been object of interest in many frameworks, for instance in theory of metric spaces of non-positive curvature, in theory of diffusion tensor imaging, in geometry of manifold of probability distributions and more generally in matrix information geometry (see for instance  \cite{Savage1982}, \cite{Skovgaard1984}, \cite{BridHaef1999}, \cite{LaLi2001}, \cite{BhaHol2006}, \cite{Bhatia2007}, \cite{MoZ2011}, \cite{NieBha2013}, \cite{Amari2016}, \cite{BarbNie2017}).
We begun the study of the trace metric on the manifold of non-singular real matrices in \cite{DoPe2015}, next we considered the same metric on the manifolds of orthogonal real matrices and of non-singular symmetric real matrices respectively in \cite{DoPe2018} and in \cite{DoPe2019}.

In the present paper we focus our attention on the elliptic isometries (i.e. the isometries having fixed points) of $(\mathcal{P}_n, g)$ and we study their loci of fixed points, providing explicit descriptions of them.

A first general result can be obtained as consequence of ordinary (but not trivial) facts of Riemannian geometry and without regarding any explicit description of such fixed loci;  precisely (Theorem \ref{Fix-Riem-manifold}):

if $\Phi$ is an elliptic isometry of $(\mathcal{P}_n,g)$, then $(Fix(\Phi), g)$ is a closed totally geodesic simply connected symmetric Riemannian submanifold of $(\mathcal{P}_n, g)$ and so $(Fix(\Phi), g)$ is a symmetric Hadamard manifold. 

An explicit description of $(Fix(\Phi), g)$ needs more careful studies of the different types of elliptic isometries.
In \cite{DoPe2019} we have already determined and described geometrically the full group of isometries of $(\mathcal{P}_n, g)$ (see also Proposition \ref{DP-torino} and Remark \ref{int-geom-isom}): 

there are four types of elliptic isometries of $(\mathcal{P}_n, g)$, consisting in

- $\Gamma_M: X \mapsto MXM^T$, the \emph{congruence} by an arbitrary non-singular real matrix $M$ (such isometries form a group acting transitively on $(\mathcal{P}_n, g)$);

- $\Gamma_M \circ \delta$ with $M \in GL_n$ and where $\delta: X \mapsto det(X)^{-2/n} X$ can be interpreted as the orthogonal symmetry with respect to the totally geodesic hypersurface $SL\mathcal{P}_n$ of matrices in $\mathcal{P}_n$ with determinant $1$;

- $\Gamma_M \circ j$ with $M \in GL_n$ and where $ j: X \mapsto X^{-1}$ can be interpreted as the central symmetry with respect to $I_n$;

- $\Gamma_M \circ j \circ \delta$ with $M \in GL_n$ and where  $j \circ \delta$ can be interpreted as the orthogonal symmetry with respect to the geodesic 
$\mathcal{R}=\{ t I_n : t \in \mathbb{R}, t>0\}$ (i.e. the geodesic through $I_n$ and orthogonal to $SL\mathcal{P}_n$).

In particular: $Fix(j) = \{I_n\}$, $Fix(\delta)= SL\mathcal{P}_n$ and $Fix(j \circ \delta) = \mathcal{R}$. 

We describe the loci of fixed points of all elliptic isometries and, as consequence, we are able to list the De Rham decompositions and the De Rham factors of all fixed loci.

This paper is organized following the different types of elliptic isometries $\Phi$: \S 3 is devoted to $\Gamma_M$, \S 4 to $\Gamma_M \circ \delta$, \S 5 to $\Gamma_M \circ j$ and \S 6 to $\Gamma_M \circ j \circ \delta$.
 
The explicit descriptions of $(Fix(\Phi), g)$ are obtained in Propositions \ref{decomp-Gamma}, \ref{decomp-Gamma-chi}, \ref{Gamma-j-Riem-prod} and \ref{Fix-Gamma-j-chi} respectively, while the complete lists of the DeRham factors are in Propositions \ref{DeRham-Gamma}, \ref{DeRham-Gamma-delta}, \ref{DeRham-Gamma-j} and \ref{DeRham-Gamma-j-delta} respectively.

Our methods involve the theory of matrices and the actions of suitable classical Lie groups.
In \S 1 we resume some facts on matrices. In particular we point out two particular canonical forms for matrices which are similar to a multiple of an orthogonal matrix: the real Jordan standard form and the real Jordan auxiliary form (see Remarks-Definitions \ref{Jordan-form} and \ref{J-tilda}); the reason is that the fixed loci are related to certain closed Lie subgroups of $GL_n$, consisting in matrices commuting with the real Jordan standard form or fixing by congruence the real Jordan auxiliary form of suitable matrices.
Some relevant properties of $(\mathcal{P}_n, g)$ and of its totally geodesic submanifolds are resumed in \S 2; these, together with some ordinary facts of Riemannian geometry, allow to obtain the general result (Theorem \ref{Fix-Riem-manifold}), quoted above.

\section{Notations and recalls on matrices}

\begin{notas}\label{notazioni}\ \\

$I_n$: the identity matrix of order $n$; 

$A^T$: the \emph{transpose} of any matrix $A$;

$M_n$ (and $Sym_n$): the vector space of the real square matrices of order $n$ (which are symmetric);

$GL_n$ (and $SL_n$): the multiplicative group of the non-singular real matrices of order $n$ (and with determinant $1$);

$\mathcal{P}_n$ (and $SL\mathcal{P}_n$): the manifold of \emph{symmetric positive definite} matrices of order $n$  (and with determinant $1$);

$\mathcal{O}_n$ (and $S\mathcal{O}_n$): the multiplicative group of real \emph{orthogonal} matrices of order $n$ (with determinant $1$);

$\mathcal{O}(p, n-p)$ (and $S\mathcal{O}_0(p, n-p)$): the \emph{generalized orthogonal group} of \emph{signature} $(p, n-p)$ (and its connected component of the identity);

$Sp_{2n}$: \emph{the real symplectic group} given by matrices $W \in GL_ {2n}$
such that

$W
\begin{pmatrix} 
 0 & I_n \\ 
-I_n &  0
\end{pmatrix}
W^T =
\begin{pmatrix} 
 0 & I_n \\ 
-I_n &  0
\end{pmatrix}
$;

$M_n(\mathbb{C})$ (and $Herm_n$): the vector space of the complex square matrices of order $n$ (which are \emph{hermitian}); 

$GL_n(\mathbb{C})$ (and $SL_n(\mathbb{C}))$: the multiplicative group of the non-singular complex matrices of order $n$ (with determinant $1$);

$\mathcal{H}_n$ (and $SL\mathcal{H}_n$): the real manifold of hermitian positive definite matrices of order $n$ (and with determinant $1$);

$U_n$ (and $SU_n)$: the multiplicative group of complex \emph{unitary} matrices of order $n$ (with determinant $1$);

$U(\mu, \nu)$ (and $SU(\mu, \nu)$): the \emph{generalized unitary} group of signature $(\mu, \nu)$ (with determinant $1$).

\smallskip

For every $A \in M_n(\mathbb{C})$, $tr(A)$ is its \emph{trace}, $A^*:=\overline{A\,}^T$ is its \emph{transpose conjugate}, $det(A)$ is its \emph{determinant} and, provided that $det(A) \ne 0$, $A^{-1}$ is its \emph{inverse} and we denote $A^{-T}=(A^T)^{-1}=(A^{-1})^T$. 

When $A \in \mathcal{P}_n$, $\sqrt{A}$ is its unique \emph{square root} contained in $\mathcal{P}_n$.

For every $\theta \in \mathbb{R}$, we denote

 $E_{\theta}:=\begin{pmatrix} 
 \cos \theta & - \sin \theta \\ 
\sin \theta &  \cos \theta
\end{pmatrix}$, $E:=E_{\pi/2} = 
\begin{pmatrix} 
0  & -1 \\ 
 1 &  0
\end{pmatrix}$, 
so $E_\theta = (\cos \theta)I_2 + (\sin \theta) E$.

If $B_1, \cdots , B_m$ are square matrices (of possible distinct orders), $B_1 \oplus \cdots \oplus B_m$ is the block diagonal square matrix with $B_1, \cdots , B_m$ on its diagonal and, for every square matrix $B$, $B^{\oplus m}$ denotes $B \oplus \dots \oplus B$ ($m$ times). The notations $(\pm I_0) \oplus B$ and $B \oplus (\pm I_0)$ simply indicate the matrix $B$.

If $\mathcal{S}_1, \dots , \mathcal{S}_m$ are sets of square matrices, then  $\mathcal{S}_1 \oplus \dots \oplus \mathcal{S}_m$ denotes the set of all matrices $B_1 \oplus \cdots \oplus B_m$ with $B_j \in \mathcal{S}_j$ for every $j$.

\smallskip

For every matrix $X \in GL_n$ we denote
\begin{center}
$\mathcal{C}_X := \{B \in GL_n : B X = X  B\}$ \ \ and \  \ 
$\mathcal{K}_X := \{ K \in GL_n : KXK^T=X\}$.
\end{center}
It is easy to check that both $\mathcal{C}_X$ and $\mathcal{K}_X$ are closed Lie subgroups of $GL_n$.

\smallskip

For any other notation and for information on the matrices, not explicitly recalled here, we refer to \cite{HoJ2013}.
\end{notas}

\begin{defi}
For every matrix $C \in GL_n(\mathbb{C})$ we denote by $\Gamma_C$, by $j$ and by $\delta$ the maps: $GL_n(\mathbb{C}) \to GL_n(\mathbb{C})$ given by 

$\Gamma_C(X) := CXC^T$ (the \emph{congruence} by $C$), 

$j(X) := X^{-1}$ and 

$\delta(X):=|det(X)|^{-2/n} X$.

The restrictions of these maps to any subset of $GL_n(\mathbb{C})$ will be still denoted by the same letters.
\end{defi}

\begin{remsdefs}
a) Two matrices $A, B \in M_n(\mathbb{C})$ are \emph{similar} if there exists a matrix $C \in GL_n(\mathbb{C})$ such that $A=CBC^{-1}$.  

When $A, B$ are real, it does not matter if $C$ is real or complex. Indeed, even if $C$ is complex, then we can find a real matrix $C'$ satisfying $A=C'B{C'}^{-1}$.

b) Two matrices $A, B \in M_n(\mathbb{C})$ are $\mathbb{K}$-\emph{congruent} with $\mathbb{K} = \mathbb{R}$ or $\mathbb{K} = \mathbb{C}$, if there is a non-singular matrix $C$, with entries in $\mathbb{K}$, such that $A=CBC^T$. 

Two real matrices, e.g. $I_2$ and $\begin{pmatrix} 
1  & 0 \\ 
0 &  -1
\end{pmatrix}$,
can be $\mathbb{C}$-\emph{congruent}, but not $\mathbb{R}$-\emph{congruent}.

c) It is known that two matrices $A, B \in GL_n(\mathbb{C})$ are $\mathbb{C}$-congruent if and only if $AA^{-T}$ and $BB^{-T} $ are similar (see for instance in \cite[Thm.\,4.5.27 p.\,295]{HoJ2013}, via the fact that $MM^{-T}$ and $M^{-T}M$ are similar for every $M \in GL_n(\mathbb{C}))$.

Assume furthermore that $A, B$ are real; if they are $\mathbb{R}$-congruent, then  $AA^{-T}$ and $BB^{-T}$ are similar too, but the converse is not generally true.

d) Finally we recall that a real matrix $A \in M_n$ is said to be \emph{normal}, if $AA^T=A^TA$.
\end{remsdefs}

\begin{thm}\label{polar-dec} (Polar decomposition, see \cite[Thm.\,7.3.1 p.\,449]{HoJ2013})

Let $A \in GL_n$. Then there exist, and are uniquely determined, $Q, Q' \in \mathcal{P}_n$ and $U, U' \in \mathcal{O}_n$ such that $A = QU = U'Q'$. 

Moreover $U=U'$, $Q= \sqrt{AA^T}$ and  $Q'= \sqrt{A^TA}$.

In particular $A \in GL_n$ is real normal if and only if $Q=Q'$, i.e. if and only if $Q$ and $U$ commute.
\end{thm}

\begin{rem}\label{dec-pol-norm}
Let $A= QU=UQ$ a real normal matrix in $GL_n$ together with its polar decompositions. From $Q=UQU^T$, we get $\sqrt{Q}=U\sqrt{Q}U^T$; therefore also $\sqrt{Q}$ commutes with $U$ and $A = \sqrt{Q}\, U \sqrt{Q}$.  Hence every real normal non-singular matrix is $\mathbb{R}$-congruent to the orthogonal matrix of its polar decomposition.
\end{rem}

\begin{thm}\label{teo-spettr-reale}

For every $A \in GL_n$ the following facts are equivalent:

i) $A=\lambda P$, where $P$ is an orthogonal matrix and $\lambda \ne 0$ is a real number;

ii) there is a matrix $Q \in \mathcal{O}_n$ 
such that 

$Q^T AQ = |\lambda| \left( I_p  \oplus E_{\theta_1}^{\oplus m_1} \oplus \cdots \oplus E_{\theta_r}^{\oplus m_r}\oplus (- I_q) \right)
$, where $\lambda \ne 0$ is a real number,

with $p, q, r \ge 0$, $m_j >0$ for $1 \le j \le r$ (if $r \ge 1$), $p+q+2m_1 + \cdots + 2m_r =n$ and  $0 < \theta_1 < \theta_2 < \cdots < \theta_r < \pi$.
\end{thm}

\begin{proof}
This is essentially the Real Spectral Theorem for matrices which are multiple of orthogonal matrices (see for instance \cite[ Cor.\,2.5.11 p.\,136--137]{HoJ2013}, except for an irrelevant change of sign), because the matrix on the right side of (ii) is multiple of an orthogonal matrix.
\end{proof}

\begin{remdef}\label{Jordan-form}

A matrix in $GL_n$ is similar to a multiple of an orthogonal matrix if and only if it is  semisimple and its eigenvalues have constant modulus.
By Theorem \ref{teo-spettr-reale}, such a matrix, $A$, is similar to a matrix of the form 
\begin{center}
$J_A:=|\lambda| \left( I_p  \oplus E_{\theta_1}^{\oplus m_1} \oplus \cdots \oplus E_{\theta_r}^{\oplus m_r}\oplus (- I_q) \right)
$ 
\end{center}

with $\lambda \ne 0$, $p, q, r \ge 0$,  $m_j >0$ for $1 \le j \le r$ (if $r \ge 1$),
$p+q+2m_1 + \cdots + 2m_r =n$ and  $0 < \theta_1 < \theta_2 < \cdots < \theta_r < \pi$.

\smallskip

Hence, for every matrix $A \in GL_n$, similar to a multiple of an orthogonal matrix, we call such matrix $J_A$ \emph{the real Jordan standard form} (shortly: RJS form) of $A$. 

We remark that the eigenvalues of $A$ (and of $J_A$) are: $|\lambda|$ with multiplicity $p$, $-|\lambda|$ with multiplicity $q$ and $|\lambda|e^{\pm i \theta_j}$ each with multiplicity $m_j$, for $j= 1, \cdots , r$ (if $r \ge 1$).

Finally, from the similarity between $A$ and $A^T$ (see for instance \cite[Thm.\,3.2.3.1, p.\,177]{HoJ2013}), we get $J_A=J_{A^T}$.
\end{remdef}

\begin{remdef}\label{J-tilda}
By technical reasons,  for every matrix $A \in GL_n$, similar to a multiple of an orthogonal matrix, we are interested in introducing another Jordan-type form, $\widetilde{J}_A$, having  the property: $(\widetilde{J}_A)^2 = J_{A^2}$.

\smallskip

By means of congruences given by orthogonal matrices,  we can arbitrarily permute the direct addends of its RJS form $J_A$.

Moreover, for \  $\Xi := {\begin{pmatrix} 
0  & 1 \\ 
 1 &  0
\end{pmatrix}}$,
we have $\Gamma_\Xi(E_{\theta})= E_{-\theta}=- E_{\pi -\theta}$. Hence for every $\theta_j \in (\dfrac{\pi}{2}, \pi)$,
up to an orthogonal congruence,
we can replace each $E_{\theta_j}$ with  $-E_{\pi- \theta_j}$. Now we  reorder the values $\theta_i \in (0, \dfrac{\pi}{2})$ together with the new values $\pi- \theta_j\in (0, \dfrac{\pi}{2})$ following the increasing order of $\theta_i$'s and of $(\pi-\theta_j)$'s; hence, after renaming them $\phi_t$, we obtain the following matrix:
\begin{center}
$
\widetilde{J}_A:= |\lambda| \left(I_p \oplus (- I_q) \oplus E_{\phi_1}^{\oplus \mu_1} \oplus (- E_{\phi_1 }^{\oplus \nu_1}) \oplus \cdots \oplus E_{\phi_h}^{\oplus \mu_h} \oplus (- E_{\phi_h}^{\oplus \nu_h}) \oplus E_{\pi/2}^{\oplus k}\right)
$
\end{center}
where $p\ge 0$ is the multiplicity of the eigenvalue $|\lambda|$, $q\ge 0$ is the multiplicity of the eigenvalue $-|\lambda|$, $k \ge 0$ is the multiplicity of the eigenvalues $\pm i |\lambda|$ and where $h \ge 0$ and $\mu_j \ge 0$ is the multiplicity of the eigenvalues $|\lambda| e^{\pm i \phi_j}$, $\nu_j \ge 0$ is the multiplicity of the eigenvalues $|\lambda| e^{\pm i (\pi -\phi_j)}$ with $\mu_j+\nu_j \ge 1$ for every $j \le h$ and $0 < \phi_1 < \cdots < \phi_h < \dfrac{\pi}{2}$.

Note that $\widetilde{J}_A= C^T A C $ for some $C \in \mathcal{O}_n$.

Finally since $(\widetilde{J}_A)^2 = \lambda^2 \left( I_{p+q} \oplus E_{2 \phi_1}^{\oplus(\mu_1 + \nu_1)} \oplus \cdots \oplus E_{2 \phi_h}^{\oplus(\mu_h + \nu_h)} \oplus (-I_{2k}) \right)$, we get that: 
\begin{center}
$(\widetilde{J}_A)^2 = J_{A^2}$.
\end{center}
We call the matrix $\widetilde{J}_A$ \emph{the real Jordan auxiliary form} (shortly: RJA form) of $A$.

As in the case of the $RJS$ forms, we have $\widetilde{J}_A = \widetilde{J}_{A^T}$.
\end{remdef}

\section{Recalls on the trace metric and a first general result}

From now on, and for the remaining part of this paper, $n$ is a fixed integer, $n\ge 2$.

\medskip

\begin{remdef}
The $C^\infty$-tensor $\gamma$ of type $(0,2)$ on $\mathcal{H}_n$, defined by 
\begin{center}
$\gamma_A(V,W) = tr(A^{-1}VA^{-1}W)$
\end{center} 
for every $A \in \mathcal{H}_n$ and for every $V, W \in T_A \mathcal{H}_n = Herm_n$,
is called \emph{trace metric}.

For convenience, the restriction of $\gamma$ to $\mathcal{P}_n$ will be denoted by $g$ (always called trace metric), so that $(\mathcal{P}_n, g)$ is a Riemannian submanifold of $(\mathcal{H}_n, \gamma)$. 

Instead we will denote again by $g$ the restriction of $g$ to any submanifold of $\mathcal{P}_n$.

For the differential-geometric properties of $(\mathcal{P}_n, g)$ and of $(\mathcal{H}_n, \gamma)$ we refer to 
\cite{Savage1982}, \cite{Skovgaard1984}, \cite{BridHaef1999}, \cite{BhaHol2006}, \cite{Bhatia2007}, \cite{MoZ2011}, \cite{NieBha2013}, \cite{DoPe2019}.
\end{remdef}

\begin{defi}
A \emph{Hadamard manifold} is a simply connected, complete, smooth Riemannian manifold without boundary and with non-positive sectional curvature. 

An isometry of a Hadamard manifold is said to be \emph{elliptic}, if it has a fixed point.

For more information on Hadamard manifolds and on their isometries we refer for instance to \cite[Lecture\,I \S\,2 and Lecture\,II \S\,6]{BaGroSch1985} and to \cite[Ch.\,1 \S\,5 and Ch.\,2 \S\,6]{Ballmann1995}.
\end{defi}

\begin{prop}\label{DP-torino} (see \cite[\S\,3 and Thm.\,4.4]{DoPe2019})

a) $(\mathcal{P}_n, g)$ is a symmetric Hadamard manifold.

b) A mapping $\Phi: (\mathcal{P}_n, g) \to (\mathcal{P}_n, g)$ is an isometry if and only if there exists a matrix $M \in GL_n$ such that one of the following cases occurs:

$\Phi(X)= \Gamma_M(X)=MXM^T$;

$\Phi(X)= (\Gamma_M \circ \delta)(X)=\dfrac{MXM^T}{det(X)^{2/n}}$; 

$\Phi(X)= (\Gamma_M \circ j)(X)= MX^{-1}M^T$; 

$\Phi(X)= (\Gamma_M \circ j \circ \delta)(X)= det(X)^{2/n}\,MX^{-1}M^T$.
\end{prop}

\begin{rem}
It is well-known that also $(\mathcal{H}_n, \gamma)$ is a symmetric Hadamard manifold and that $(\mathcal{P}_n, g)$ is a totally geodesic Riemannian submanifold of $(\mathcal{H}_n, \gamma)$.

The description of the isometries of $\mathcal{H}_n$ endowed with a class of metrics which includes $\gamma$, is given in \cite[Thm.\,3]{Moln2015}. 
As already remarked in \cite{DoPe2019}, from the comparison  with the previous result, it follows that every isometry of $(\mathcal{P}_n, g)$ is the restriction of an isometry of $(\mathcal{H}_n, \gamma)$.
\end{rem}

\begin{defi}
Let $G$ be a closed subgroup of $GL_n$.

$G$ is said to be \emph{reductive}, if $A^T \in G$ as soon as $A \in G$. 

$G$ is said to be
 \emph{algebraic}, if there is a finite system of polynomials (in the entries of $M_n$) such that $G$ is the intersection of $GL_n$ with the set of common zeroes of this system. 
\end{defi}

\begin{prop}\label{Gruppi-alg-riduttivi}(see \cite[Thm.10.58]{BridHaef1999}) \\
Let $G$ be an reductive subgroup of $GL_n$ satisfying the following property:

(*) \ \ \ if $X \in Sym_n$ and $e^X \in G$, then $e^{s X} \in G$ for every $s \in \mathbb{R}$. 

Then

i) $(G \cap \mathcal{P}_n, g)$ is a totally geodesic submanifold of $(\mathcal{P}_n, g)$;

ii) $G \cap \mathcal{P}_n$ is the orbit of $I_n$ under the action of $G$ by congruence, so that $G \cap \mathcal{P}_n$ is diffeomorphic to $G/(G\cap O_n)$;

iii) $(G \cap \mathcal{P}_n, g)$ is a symmetric Riemannian manifold with non-positive sectional curvature.
\end{prop}

\begin{rem}

Let $G$ be a closed sugbroup of $GL_n$, then $G$ satisfies the condition (*) of Proposition \ref{Gruppi-alg-riduttivi} if and only if $G_0$ (the connected component of the identity of $G$) satisfies (*).
Moreover any algebraic subgroup of $GL_n$  satisfies the condition (*) (see \cite[Lemma 10.59]{BridHaef1999}).

\end{rem}

\begin{rem}\label{rho}
The mapping 
$\rho: \mathbb{C} \to M_2$, given by $\rho(z) =  Re(z) I_2 + Im(z) E$,
 is a monomorphism of $\mathbb{R}$-algebras between $\mathbb{C}$ and $M_2$. 
Note that $\rho({\overline{z}}) = \rho(z)^T$
and that $\rho(z) \in GL_2$ as soon as $z \ne 0$. 

More generally, for any $h \ge 1$, we denote again by $\rho$ the mapping: $M_h(\mathbb{C}) \to M_{2h}$, which maps the $h \times h$ complex matrix $Z=(z_{ij})$ to the $(2h) \times (2h)$ block real matrix  $(\rho(z_{ij}))$, having $h^2$ blocks of order $2 \times 2$. 

In literature there are other ways, essentially equivalent to $\rho$, to embed $M_h(\mathbb{C})$ into $M_{2h}$ (see for instance \cite[Prop.\,2.12]{deGoss2006}). It seem to us that the mapping $\rho$, used here, is more useful for the purposes of present paper.

Standard arguments show that $tr(\rho(Z)) = 2 Re(tr(Z)))$, $det(\rho(Z)) = |det(Z)|^2$ and that $\rho$ is a monomorphism of $\mathbb{R}$-algebras, whose restriction to $GL_h(\mathbb{C})$ has image into $GL_{2h}$ and it is a monomorphism of Lie groups.

We have: $\rho(Z^*) = \rho(Z)^T$ and, so, the restriction of $\rho$ to $U_h$ is again a monomorphism of Lie groups
and $\rho (U_h)=  \rho (GL_h(\mathbb{C})) \cap S\mathcal{O}_{2h}$; analogously the restriction of $\rho$ to $Herm_h$ has image into $Sym_{2h}$ and it is  a monomorphism of $\mathbb{R}$-vector spaces. 

Finally $\rho$ maps injectively $\mathcal{H}_h$ into $\mathcal{P}_{2h}$. 
Indeed $\rho(ZZ^*)= \rho(Z) \rho(Z)^T$.

Moreover, for $A \in \mathcal{H}_h$ and $Z, W \in Herm_h$, we get: 

$g_{\rho(A)}(d\rho(Z), d\rho(W)) = tr(\rho(A)^{-1} \rho(Z) \rho(A)^{-1} \rho(W)) = tr(\rho(A^{-1}Z A^{-1}W))=$

$=2Re\, tr(A^{-1}Z A^{-1}W) = 2 Re\,\gamma_A(Z,W) = 2\gamma_A(Z,W)$.

Hence the restriction of $\rho$ from $(\mathcal{H}_{h}, 2\gamma$) into $(\mathcal{P}_{2h}, g)$ is an isometry onto its image $\rho(\mathcal{H}_{h}) =\rho(GL_h(\mathbb{C})) \cap \mathcal{P}_{2h}$.
\end{rem}

\begin{defi}
For every isometry $\Phi$ of $(\mathcal{P}_n, g)$ we denote by $Fix(\Phi)$ the set of points of $\mathcal{P}_n$ fixed by $\Phi$.
\end{defi}

\begin{thm}\label{Fix-Riem-manifold}
If $\Phi$ is an elliptic isometry of $(\mathcal{P}_n,g)$, then $(Fix(\Phi), g)$ is a closed totally geodesic simply connected symmetric Riemannian submanifold of $(\mathcal{P}_n, g)$ and so $(Fix(\Phi), g)$ is a symmetric Hadamard manifold. 
\end{thm}

\begin{proof}
Each connected component of $Fix(\Phi)$ is a closed totally 
geodesic submanifold  of $\mathcal{P}_n$ by \cite[Thm.\,5.1 p.\,59]{Koba1995}. 
By completeness of $\mathcal{P}_n$, points in different components of $Fix(\Phi)$ should be mutually \emph{cut points} (see \cite[Cor.\,5.2 p.\,60]{Koba1995} and, for more information on cut points, \cite[Ch.\,VIII \S\,7]{KoNo2}).
Now, since $\mathcal{P}_n$ is a Hadamard manifold, by Cartan-Hadamard Theorem (see for instance \cite[Thm.\,22 p.\,278]{O'Neill1983} and \cite[Lecture\,1 \S 2]{BaGroSch1985}) any two points are joined by a unique minimizing geodesic. Hence $\mathcal{P}_n$ has no cut points and therefore $Fix(\Phi)$ is connected too.
Moreover $(Fix(\Phi), g)$ is complete, its curvature is non-positive and it has no non-trivial geodesic loop, because it is closed and totally geodesic in the Hadamard manifold $(\mathcal{P}_n,g)$. Hence, by \cite[Cor.\,9.2.8]{Burago2001}, $(Fix(\Phi), g)$ is simply connected and so it is a Hadamard manifold. Finally $(Fix(\Phi), g)$ is symmetric, because it is a totally geodesic Riemannian submanifold of the symmetric Riemannian manifold $(\mathcal{P}_n,g)$.
\end{proof}

\begin{remdef}
The previous Theorem implies that, if $\Phi$ is an elliptic isometry of $(\mathcal{P}_n,g)$, then $(Fix(\Phi), g)$ has a \emph{De Rham decomposition} into a Riemannian product: one of its factors (called \emph{flat } or \emph{Euclidean factor}) may be isometric to some Euclidean space $\mathbb{R}^m$, while the other factors are irreducible symmetric Hadamard manifolds. Such decomposition is unique up to isometries and permutations of its factors (see for instance \cite[Ch.\,IV, \S\,6]{KoNo1}) and each factor is an \emph{Einstein} manifold (see for instance \cite[Note 10.83, p.\,298]{Bes1987}). The irreducible simply connected symmetric spaces are classified and the complete list is for instance in  \cite[pp.\,311--317]{Bes1987} and in \cite[pp.306--308]{BerConOl2003}.

In this paper we determine the De Rham decomposition of every such $(Fix(\Phi), g)$.
\end{remdef}

\begin{rem}\label{int-geom-isom}
In \cite[Rem.\,4.6]{DoPe2019} we described geometrically the particular isometries  $\delta$, $j$ and $j \circ \delta$ of $(\mathcal{P}_n, g)$ as follows:

- $\delta$ is the orthogonal symmetry with respect to the hypersurface $SL\mathcal{P}_n$;

- $j$ is the symmetry with respect to $I_n$;

- $j \circ \delta = \delta \circ j$ is the orthogonal symmetry with respect to the geodesic 

$\mathcal{R}=\{ t I_n : t \in \mathbb{R}, t>0\}$ (i.e. the geodesic through $I_n$ and orthogonal to $SL\mathcal{P}_n$).

In particular $Fix(j) = \{I_n\}$, $Fix(\delta)= SL\mathcal{P}_n$ and $Fix(j \circ \delta) = \mathcal{R}$.

Hence we want to describe explicitly the fixed loci, when the previous isometries are composed with congruences.
\end{rem}

\begin{rem}\label{det-P-j-chi}
Let $M \in GL_n$. Then $Fix(\Gamma_M \circ \delta)$ and $Fix(\Gamma_M \circ j)$ are both contained in
$\{ P \in \mathcal{P}_n : det(P) = |det(M)| \}$.

This follows by computing the determinants from the equalities: $\dfrac{MPM^T}{det(P)^{2/n}}=P$ and

 $MP^{-1}M^T=P$.
\end{rem}

\section{The fixed points of the isometries $\Gamma_M$}

\begin{prop}\label{class-ellitt-Gamma}
Let $M \in GL_n$ and let us consider the isometry $\Gamma_M$ of $\mathcal{P}_n$.
The following facts are equivalent:

a) $\Gamma_M$ is elliptic;

b) $M$ is similar to an orthogonal matrix; 
 
c) $M$  is semi-simple and its eigenvalues have modulus $1$.

If this is the case, the RJS form of $M$ is of the type:

 $J_M=I_p  \oplus E_{\theta_1}^{\oplus m_1} \oplus \cdots \oplus E_{\theta_r}^{\oplus m_r}\oplus (- I_q)$ 

(with $p, q, r \ge 0$, $m_j > 0$ for every possible $j$, $p+q+2m_1 + \cdots + 2m_r =n$ and $0 < \theta_1 < \theta_2 < \cdots < \theta_r < \pi$).
\end{prop}

\begin{proof}
The equivalence between (b) and (c) and the assertion about the RJS form are in \ref{Jordan-form}.

Hence it suffices to prove the equivalence between (a) and (b). 

If $P \in \mathcal{P}_n$ verifies: $MPM^T=P$, then, for any $C \in GL_n$ satisfying $P=CC^T$, we get:
$(C^{-1}MC) (C^{-1}MC)^T =(C^{-1}MC)(C^TM^T C^{-T}) = I_n$, i.e. $C^{-1}MC \in \mathcal{O}_n$.

For the converse, let $C$ be any non-singular matrix such that $C^{-1}MC \in \mathcal{O}_n$, then: 

$(C^{-1}MC)(C^TM^T C^{-T}) = I_n$ and finally: $M(CC^T)M^T = CC^T$. This allows to conclude.
\end{proof}

\begin{prop}\label{Fix-Gamma}
Let $M \in GL_n$ and assume that the isometry $\Gamma_M$ of $\mathcal{P}_n$ is elliptic. Then

i) $Fix(\Gamma_M) =\{CC^T : C\in GL_n, C^{-1}MC \in \mathcal{O}_n\}\\
\hspace*{.7in}          = \{FF^T : F\in GL_n, F J_M F^{-1} = M\};$ 
           
ii) for any fixed matrix $F_0 \in GL_n$ such that $F_0 J_M F_0^{-1} = M$, we have:

$Fix(\Gamma_M) = \{F_0AA^TF_0^T : A \in \mathcal{C}_{J_M}\}$.

\end{prop}

\begin{proof}

The first equality of (i) has been essentially obtained in verifying the equivalence between (a) and (b) of Proposition \ref{class-ellitt-Gamma}. For the second equality of (i), it suffices to note that $J_M$ is orthogonal and that if $C^{-1}MC \in \mathcal{O}_n$ then $C^{-1}MC = QJ_MQ^T$ for some $Q \in \mathcal{O}_n$, hence $F=CQ$ satisfies $FF^T=CC^T$ and $F^{-1}MF \in \mathcal{O}_n$.

Now we prove (ii).
Let $P \in \mathcal{P}_n$ be a fixed point of $\Gamma_M$. By (i), we can choose a matrix $F_0 \in GL_n$ such that $F_0J_M F_0^{-1} =M$ and $F_0F_0^T=P$. Let $F\in GL_n$ any other matrix such that $FJ_M F^{-1} =M$ and $FF^T=P$, then an easy computation shows that the matrix $A:=F_0^{-1} F$ satisfies $AJ_M=J_MA$, i. e. $A \in \mathcal{C}_{J_M}$. Therefore $P=FF^T= (F_0A)(F_0A)^T=F_0AA^TF_0^T$ with $A=F_0^{-1} F \in \mathcal{C}_{J_M}$.

Conversely let $F = F_0 A$ with $A \in \mathcal{C}_{J_M}$. Then $F J_M F^{-1} = F_0 A J_M A^{-1} F_0^{-1} = F_0 J_M F_0^{-1} = M$. Then, by assertion (i), $FF^T=(F_0 A)(F_0 A)^T=F_0AA^TF_0^T$ is a fixed point of $\Gamma_M$ in $\mathcal{P}_n$.
\end{proof}

\begin{lemma}\label{lemma-su-J}
Let $0 <\theta_1 < \theta_2 < \cdots < \theta_r < \pi$ be real numbers
and let 

$J:=I_p \oplus E_{\theta_1}^{\oplus m_1} \cdots \oplus E_{\theta_r}^{\oplus m_r} \oplus (-I_q)$ with $p, q, r \ge 0$, $m_j > 0$ for every possible $j$, 

$p+q+ 2m_1 + \cdots +2m_r=n$.

Then the set of matrices of $M_n$, commuting with $J$, is the vector space 
\begin{center}
$M_p \oplus \rho (M_{m_1}(\mathbb{C})) \oplus \cdots \oplus\rho ( M_{m_r}(\mathbb{C})) \oplus M_q$.
\end{center}

In particular the Lie group of non-singular matrices, commuting with $J$, is 
\begin{center}
$GL_p \oplus \rho (GL_{m_1}(\mathbb{C})) \oplus \cdots \oplus \rho (GL_{m_r}(\mathbb{C})) \oplus GL_q$,
\end{center}
 which is an algebraic reductive subgroup of $GL_n$.
\end{lemma}

\begin{proof}
In order to simplify the next computations we denote $\sigma = m_1 + \cdots + m_r$, $F_0=I_p$, $F_{\sigma +1} =-I_q$, $F_1 = F_2 = \cdots = F_{m_1} = E_{\theta_1}, \  F_{m_1 +1} = \cdots = F_{m_1+m_2} = E_{\theta_2},\cdots   , F_{m_1 + \cdots + m_{r-1} +1} = \cdots  = F_{\sigma} = E_{\theta_r}$, so that
$J= F_0 \oplus F_1  \oplus \cdots \oplus F_\sigma \oplus F_{\sigma+1}$.

Let $A \in M_n$. We write $A$ in blocks: $A = (A_{ij})$ with $i,j = 0, \cdots , \sigma +1$, $A_{ij} \in M_2$ for ever $1 \le i, j \le \sigma$,  $A_{00} \in M_p$,  $A_{\sigma+1, \sigma +1} \in M_q$ and with the remaining matrices of obvious orders, in line with the decomposition in blocks.

The condition $AJ=JA$ is equivalent to 

(**) \ \ \ $A_{ij} F_j = F_i A_{ij}$, for every $i,j =0, \cdots , \sigma +1$.

Easy computations show directly that $A_{0, \sigma +1} = A_{\sigma +1,0}=0$ and that $A_{00}$ and $A_{\sigma +1, \sigma +1}$ are generic matrices in $M_p$ and $M_q$ respectively.

When $i \in \{0, \sigma +1\}$ and $1\le j \le \sigma$ or $j \in \{0, \sigma +1\}$  and $1\le i \le \sigma$, the condition (**) implies that $A_{ij}=0$.
Indeed, when $i=0$ and $1\le j \le \sigma$, (**) gives: $A_{0j}(F_j-I_2)=0$ and we conclude since $det(F_j-I_2)>0$. Analogously we can conclude in the other three cases.

Now, for $1 \le i, j \le \sigma$, (**) can be written as 

$A_{ij} 
 \begin{pmatrix} 
 \cos\varphi & - \sin\varphi \\ 
\sin\varphi & \cos\varphi
 \end{pmatrix}=
  \begin{pmatrix} 
 \cos\psi & -\sin\psi \\ 
\sin\psi & \cos\psi
 \end{pmatrix} A_{ij}
$, 
with $\varphi, \psi \in \{\theta_1, \cdots , \theta_r \}$.

This gives a homogeneous linear system $4 \times 4$ with unknowns the entries of the matrix $A_{ij}$, whose determinant is 

$[(\cos \psi -\cos \varphi)^2 - \sin^2 \psi + \sin^2 \varphi]^2 + 4 (\cos \psi -\cos \varphi)^2\sin^2\psi$. This expression is non-zero except for $\varphi = \psi$.

Hence, if $\varphi \ne \psi$, then $A_{ij}=0$.

If $\varphi = \psi$, then the rank of the matrix associated to the system is $2$ and, so, the space of its solutions is the $\mathbb{R}$-vector space spanned by $I_2$ and $E$, i.e. it is $\rho(\mathbb{C})$.

This concludes the first part of the statement. The second part follows easily from arguments about the non-singularity of the matrices.
\end{proof}

\begin{prop}\label{decomp-Gamma}
Let $M \in GL_n$ such that $\Gamma_M$ is elliptic with eigenvalues $1$ of multiplicity $p$, $-1$ of multiplicity $q$ , $e^{\pm \theta_1}$ both with multiplicity $m_1$, $\cdots$,  up to $e^{\pm \theta_r}$ both with multiplicity $m_r$, with $p, q, r \ge 0$, $m_j > 0$ for $1 \le j \le r$, 
$p+q+ 2m_1 + \cdots +2m_r=n$ and
$0 < \theta_1 < \cdots < \theta_r <\pi$. Fix $F_0 \in GL_n$ such that $F_0 J_M F_0^{-1} = M$.   Then
\begin{center}
$\Gamma_{F_0^{-1}}(Fix(\Gamma_M))= \mathcal{P}_p \oplus \rho (\mathcal{H}_{m_1}) \oplus \cdots \oplus \rho (\mathcal{H}_{m_r}) \oplus \mathcal{P}_q$.
\end{center}
Hence $\left( Fix(\Gamma_M), g \right)$
 is a closed simply connected totally geodesic symmetric Riemannian submanifold of $(\mathcal{P}_n, g)$ isometric to the Riemannian product 
\begin{center}
$(\mathcal{P}_p, g) \times (\mathcal{H}_{m_1}, 2\gamma) \times \cdots \times (\mathcal{H}_{m_r}, 2\gamma) \times (\mathcal{P}_q, g)$.
\end{center}
\end{prop}

\begin{proof}
From the previous Lemma \ref{lemma-su-J}, the set of matrices $AA^T$ with $A \in \mathcal{C}_{J_M}$ is: 

$\mathcal{P}_p \oplus \rho (\mathcal{H}_{m_1}) \oplus \cdots \oplus \rho (\mathcal{H}_{m_r}) \oplus \mathcal{P}_q$.
Hence, from the Proposition \ref{Fix-Gamma}, $Fix(\Gamma_M) = \\ = \Gamma_{F_0}(\mathcal{P}_p \oplus \rho (\mathcal{H}_{m_1}) \oplus \cdots \oplus \rho (\mathcal{H}_{m_r}) \oplus \mathcal{P}_q)$. Since $\Gamma_{F_0}$ is an isometry of $(\mathcal{P}_n,g)$ and by remarking that the Riemannian manifold $(\mathcal{P}_p \oplus \rho (\mathcal{H}_{m_1}) \oplus \cdots \oplus \rho (\mathcal{H}_{m_r}) \oplus \mathcal{P}_q, g)$ is canonically isometric to
$(\mathcal{P}_p, g) \times (\mathcal{H}_{m_1}, 2\gamma) \times \cdots \times (\mathcal{H}_{m_r}, 2\gamma) \times (\mathcal{P}_q, g)$, we can conclude by Proposition \ref{Gruppi-alg-riduttivi}, taking into account that 

$\mathcal{P}_p \oplus \rho (\mathcal{H}_{m_1}) \oplus \cdots \oplus \rho (\mathcal{H}_{m_r}) \oplus \mathcal{P}_q = G \cap \mathcal{P}_n$, where $G$ is the algebraic reductive subgroup of $GL_n$ defined by 
$G:=GL_p \oplus \rho (GL_{m_1}(\mathbb{C})) \oplus \cdots \oplus \rho (GL_{m_r}(\mathbb{C})) \oplus GL_q$.
\end{proof}

\begin{rem}
The values $p, q, r$ in Proposition \ref{decomp-Gamma} are non-negative and not all zero. When some of them vanishes, the Riemannian product in this Proposition must be intended in a suitable (but obvious) way. 
For instance, if $p=0$ (or $q=0$), the factor $(\mathcal{P}_p, g)$ (or $(\mathcal{P}_p, g)$) does not appear and if $r=0$, no factor $(\mathcal{H}_{m_j}, 2\gamma)$ appears.
Analogous remarks can be done about next Propositions  \ref{decomp-Gamma-chi}, \ref{Gamma-j-Riem-prod} and \ref{Fix-Gamma-j-chi}.
\end{rem}

\begin{rem}\label{DeRham-casi}
For every $m \ge 1$, it is well-known that $(\mathcal{P}_m, g)$ isometric to the Riemannian product of $(SL\mathcal{P}_m, g)$ with $\mathbb{R}$ and that $SL\mathcal{P}_m$ is diffeomorphic to $SL_m/S\mathcal{O}_m$ (see for instance \cite[p.\, 325]{BridHaef1999}). 
Hence, remembering that $SL\mathcal{P}_1$ is a point, we get that the De Rham factors of $(\mathcal{P}_m, g)$ are $SL_m/S\mathcal{O}_m$ and $\mathbb{R}$ when $m \ge 2$, while, when $m=1$, $\mathbb{R}$ is the unique factor (see for instance \cite[p.306]{BerConOl2003}).

Analogously it is easy to show that $(\mathcal{H}_{m}, 2 \gamma)$ is isometric to the Riemannian product of $(SL\mathcal{H}_m, 2 \gamma)$ with $\mathbb{R}$ and that $SL\mathcal{H}_m$ is diffeomorphic to $SL_m(\mathbb{C})/SU_m$. Hence, as above, we get that the De Rham factors of $(\mathcal{H}_{m}, 2 \gamma)$ are $SL_m(\mathbb{C})/SU_m$ and $\mathbb{R}$ when $m \ge 2$, while, when $m=1$, $\mathbb{R}$ is again the unique factor (see again for instance \cite[p.308]{BerConOl2003}).

We denote by $r'=r'(p,q,r)$ the quantity: $r' = r$ if $p=q = 0$, $r'=r+1$ if either $p=0$ or $q=0$ (but not both zero) and $r' = r +2$ if $p$ and $q$ are both non-zero. Note that $r' \ge 1$.  Now we can conclude with the following result
\end{rem}

\begin{prop}\label{DeRham-Gamma}
Let $M \in GL_n$ such that $\Gamma_M$ is elliptic, let $p, q, r, m_1, \cdots , m_r$ be as in Proposition \ref{decomp-Gamma} and $r'$ be as in Remark \ref{DeRham-casi}. Then, up to isometries,  the De Rham factors of $\left( Fix(\Gamma_M), g \right)$ are:

a) $\mathbb{R}^{r'}$;

b) $SL_p/S\mathcal{O}_p$, if $p \ge 2$;

c) $SL_q/S\mathcal{O}_q$, if $q \ge 2$;

d) $SL_{m_j}(\mathbb{C})/SU_{m_j}$ for all indices $j= 1, \cdots , r$ such that $m_j \ge 2$, if  $r \ge 1$.
\end{prop}

\section{The fixed points of the isometries $\Gamma_M \circ \delta$}

\begin{prop}\label{class-ellitt-Gamma-chi}
Let $M \in GL_n$ and let us consider the isometry $\Gamma_M \circ \delta$ of $\mathcal{P}_n$. 

\smallskip

1) The following fact are equivalent

a) $\Gamma_M \circ \delta$ is elliptic; 

b) $M$  is semi-simple and all of its eigenvalues have the same modulus.

If this is the case, the RJS form of $M$ is of the type
\begin{center}
$J_M=|det(M)|^{1/n}[I_p \oplus E_{\theta_1}^{\oplus m_1} \oplus \cdots \oplus E_{\theta_r}^{\oplus m_r} \oplus (- I_q)]$
\end{center}
(with $p, q, r \ge 0$, $m_j>0$ for every possible $j$, $p+q+2m_1 + \cdots + 2m_r =n$ 

and $0 < \theta_1 < \theta_2 < \cdots < \theta_r < \pi$).

\smallskip

2) If $\Gamma_M \circ \delta$ is elliptic, then 
\begin{center}
$Fix(\Gamma_M \circ \delta) = Fix(\Gamma_{[\frac{M}{|det(M)|^{1/n}}]}) \cap \{P \in \mathcal{P}_n : det(P) = |det(M)|\}$.
\end{center}
\end{prop}

\begin{proof}
We first prove part (2). Assume that $\dfrac{MPM^T}{det(P)^{2/n}} =P$. By Remark \ref{det-P-j-chi}, we get:
$|det(M)|=det(P)$ and therefore $(\dfrac{M}{|det(M)|^{1/n}}) P \big( \dfrac{M}{|det(M)|^{1/n}}\big)^T = P$.

The other inclusion follows easily in a similar way.

Part (1) follows from (2) and from Proposition \ref{class-ellitt-Gamma}, since
 $Fix(\Gamma_{[\frac{M}{|det(M)|^{1/n}}]}) \ne \emptyset$ implies $Fix(\Gamma_{[\frac{M}{|det(M)|^{1/n}}]}) \cap \{P \in \mathcal{P}_n : det(P) = |det(M)|\} \ne \emptyset$ too, because every congruence is linear.
\end{proof}

\begin{prop}\label{altraFixGamma-delta}
Let $M \in GL_n$, assume that the isometry $\Gamma_M \circ \delta$ of $\mathcal{P}_n$ is elliptic and fix a matrix $F_0 \in GL_n$ such that $F_0 J_M F_0^{-1} = M$ and such that $|det(F_0)| = \sqrt{|det(M)|}$.
Then 
\begin{center}
$Fix(\Gamma_M \circ \delta) = \{F_0AA^TF_0^T : A \in \mathcal{C}_{J_M} , \, det(AA^T)=1\}$.
\end{center}
\end{prop}

\begin{proof}
If $M=F_0 J_M F_0^{-1}$, then $\frac{M}{|det(M)|^{1/n}}= F_0 J_{[\frac{M}{|det(M)|^{1/n}}]} F_0^{-1}$. From the Proposition \ref{Fix-Gamma}, we get that $F_0 F_0^T$ is a fixed point of $\Gamma_{[\frac{M}{|det(M)|^{1/n}}]}$ with $det(F_0 F_0^T)= |det(M)|$. Hence, again by Proposition \ref{Fix-Gamma} and by Proposition \ref{class-ellitt-Gamma-chi}, we obtain:

 $Fix(\Gamma_M \circ \delta) = \{ F_0AA^TF_0^T : A \in \mathcal{C}_{J_M}, \, det(AA^T)=1 \}$, since the matrices commuting with $J_M$ are precisely the matrices commuting with $J_{[\frac{M}{|det(M)|^{1/n}}]}$.
\end{proof}

\begin{nota}
We denote by $\mathbf{S} \big( (\mathcal{P}_p, g) \times (\mathcal{H}_{m_1}, 2\gamma) \times \cdots \times (\mathcal{H}_{m_r}, 2\gamma) \times (\mathcal{P}_q, g) \big)$ the Riemannian submanifold of the  Riemannian product \\
$(\mathcal{P}_p, g) \times (\mathcal{H}_{m_1}, 2\gamma) \times \cdots \times (\mathcal{H}_{m_r}, 2\gamma) \times (\mathcal{P}_q, g)$, consisting in elements\\ 
$(A, B_1, \cdots, B_r, C) \in  \big( \mathcal{P}_p \times \mathcal{H}_{m_1}  \times \cdots \times \mathcal{H}_{m_r} \times \mathcal{P}_q \big)$ such that\\ 
$det(A) \, [det(B_1)]^2 \, \cdots \, [det(B_r)]^2 \, det(C) = 1$.
\end{nota}

\begin{prop}\label{decomp-Gamma-chi}
Let $M \in GL_n$ such that $\Gamma_M \circ \delta$ is elliptic and set $\eta = |det(M)|^{1/n}$ so the eigenvalues of $M$ are: $\eta$  of multiplicity $p$, $-\eta$ of multiplicity $q$ , $\eta \, e^{\pm \theta_1}$ both with multiplicity $m_1$. $\cdots$, up to $\eta \, e^{\pm \theta_r}$ both with multiplicity $m_r$, with $p, q, r \ge 0$, $m_j > 0$ for $1 \le j \le r$, 
$p+q+ 2m_1 + \cdots +2m_r=n$
and $0 < \theta_1 < \cdots < \theta_r <\pi$. Fix $F_0 \in GL_n$ such that $F_0 J_M F_0^{-1} = M$ and such that $|det(F_0)| = \sqrt{|det(M)|}$. Then
\begin{center}
$\Gamma_{F_0^{-1}}(Fix(\Gamma_M \circ \delta))= (\mathcal{P}_p \oplus  \rho (\mathcal{H}_{m_1}) \oplus \cdots \oplus \rho (\mathcal{H}_{m_r}) \oplus \mathcal{P}_q) \cap SL\mathcal{P}_n$.
\end{center}
Hence $(Fix(\Gamma_M \circ \delta), g)$
is a closed simply connected totally geodesic symmetric Riemannian submanifold of $(\mathcal{P}_n, g)$ isometric to
\begin{center}
$\mathbf{S} \big( (\mathcal{P}_p, g) \times (\mathcal{H}_{m_1}, 2\gamma) \times \cdots \times (\mathcal{H}_{m_r}, 2\gamma) \times (\mathcal{P}_q, g) \big)$.
\end{center}
\end{prop}

\begin{proof}
It is analogous to the proof of Proposition \ref{decomp-Gamma} via Propositions \ref{altraFixGamma-delta} and \ref{Gruppi-alg-riduttivi}, taking into account that
$ (\mathcal{P}_p \oplus  \rho (\mathcal{H}_{m_1}) \oplus \cdots \oplus \rho (\mathcal{H}_{m_r}) \oplus \mathcal{P}_q) \cap SL\mathcal{P}_n = G' \cap \mathcal{P}_n$ where $G'$ is the algebraic reductive subgroup of $GL_n$ defined by\\ 
$G':=\big( GL_p \oplus \rho (GL_{m_1}(\mathbb{C})) \oplus \cdots \oplus \rho (GL_{m_r}(\mathbb{C})) \oplus GL_q \big) \cap SL_n$ and that $(G' \cap \mathcal{P}_n, g)$ is isometric to 
$\mathbf{S} \big( (\mathcal{P}_p, g) \times (\mathcal{H}_{m_1}, 2\gamma) \times \cdots \times (\mathcal{H}_{m_r}, 2\gamma) \times (\mathcal{P}_q, g) \big)$.
\end{proof}

\begin{prop}\label{DeRham-Gamma-delta}
Let $M \in GL_n$ such that $\Gamma_M \circ \delta$ is elliptic, let $p, q, r, m_1, \cdots , m_r$ be as in Proposition \ref{decomp-Gamma-chi} and $r'$ be as in Remark \ref{DeRham-casi}. Then, up to isometries,  the De Rham factors of $\left( Fix(\Gamma_M \circ \delta), g \right)$ are:

a) $\mathbb{R}^{r'-1}$, if $r' \ge 2$;

b) $SL_p/S\mathcal{O}_p$, if $p \ge 2$;

c) $SL_q/S\mathcal{O}_q$, if $q \ge 2$;

d) $SL_{m_j}(\mathbb{C})/SU_{m_j}$ for all indices $j= 1, \cdots , r$ such that $m_j \ge 2$, if $r \ge 1$.
\end{prop}

\begin{proof}
We proof this result under the assumption that $p, q, r$ are all non-zero; in this case $r'-1 = r+1$. The proof can be easily adapted to the cases when some of these values vanishes.

It is a standard computation to show that the mapping defined by: 

$F(A, B_1, \cdots , B_r, C)=
\big( \dfrac{A}{det(A)^{1/p}}, \dfrac{B_1}{det(B_1)^{1/m_1}} , \cdots , \dfrac{B_r}{det(B_r)^{1/m_r}}, \dfrac{C}{det(C)^{1/q}},$

\smallskip

\ \ \ \ \ \ \ \ \ \ \ \ \ \ \ \ \ \ \ \ \ \ \ \ \ \ \ \ \ \ \ \ \ \ \ \ \ \ \ \ \ \ \ \ \ \ \ \ \   $\ln(det(A) ),
2\ln(det(B_1)), \cdots ,
2\ln(det(B_r)) \big)$

is an isometry from $\mathbf{S} \big( (\mathcal{P}_p, g) \times (\mathcal{H}_{m_1}, 2\gamma) \times \cdots \times (\mathcal{H}_{m_r}, 2\gamma) \times (\mathcal{P}_q, g) \big)$ onto

$(SL\mathcal{P}_p, g) \times \prod_{i=1}^r (SL\mathcal{H}_{m_i}, 2 \gamma) \times (SL\mathcal{P}_q, g) \times (\mathbb{R}^{r+1}, \tau)$,

where $\tau=\dfrac{dx_0^2}{p}+ \sum_{i=1}^r \dfrac{dx_i^2}{2 m_i} + \dfrac{(\sum_{i=0}^r d x_i)^2}{q}$ is a flat Riemannian metric on $\mathbb{R}^{r+1}$

and that the inverse of $F$ is: $F^{-1}(\alpha, \beta_1, \cdots , \beta_r, \gamma,  t_0, t_1, \cdots , t_r))=\\
\big( e^{t_0/p} \, \alpha, e^{t_1/2m_1}\,  \beta_1, \cdots , e^{t_r/2m_r}\,  \beta_r, e^{- (\sum_{i=0}^r t_i)/q} \, \gamma \big)
$.

Since $\tau$ is a flat metric, $(\mathbb{R}^{r+1}, \tau)$ is isometric to the Euclidean space $\mathbb{R}^{r+1}$.

This allows to conclude arguing as in Remark \ref{DeRham-casi}.
\end{proof}

\section{The fixed points of the isometries $\Gamma_M \circ j$}

\begin{prop}\label{class-ellitt-Gamma-j}
Let $M\in GL_n$ and let us consider the isometry $\Gamma_M \circ j$ of $\mathcal{P}_n$.
The following facts are equivalent:

a) $\Gamma_M \circ j$ is elliptic;

b) $M$ and $M^{-T}$ are $\mathbb{R}$-congruent via a positive definite matrix;

c) $M$ is $\mathbb{R}$-congruent to an orthogonal matrix; 

d) $M$ is $\mathbb{R}$-congruent to a normal matrix; 

e) $\Gamma_{M M^{-T}}$ is elliptic;

f) $MM^{-T}$ is semi-simple with eigenvalues of modulus $1$;

g) $MM^{-T}$ is similar to an orthogonal matrix.
\end{prop}

\begin{proof} 
The equivalence (a) $\Leftrightarrow$ (b) follows by remarking that $M(CC^T)^{-1}M^T = CC^T$ is equivalent to $M = CC^T M^{-T}CC^T$.

The equivalence: (a) $\Leftrightarrow$ (c) follows by remarking that $M(CC^T)^{-1}M^T=CC^T$ if and only if
 $(C^{-1} M C^{-T})(C^{-1} M C^{-T})^T=I_n$  if and only if $C^{-1} M C^{-T} \in \mathcal{O}_n$.

We get (d) $\Rightarrow$ (c) from Remark \ref{dec-pol-norm}, while (c) $\Rightarrow$ (d) is trivial.

Now note that, if $H \in GL_n$, we have:
$(HMH^T)(HM^T H^T) = (HM^T H^T)(HMH^T)$ if and only if 
$(M^{-T} M) H^T H (M^{-T} M)^T=  H^TH$ if and only if $\Gamma_{M^{-T} M}$ has $H^TH$ as fixed point in $\mathcal{P}_n$. Since $M^{-T}M$ and $MM^{-T}$ are similar, from Proposition \ref{class-ellitt-Gamma}, we get the equivalence (d) $\Leftrightarrow$ (e).

Finally (e) $\Leftrightarrow$ (f)  $\Leftrightarrow$ (g) is proved in Proposition \ref{class-ellitt-Gamma}.
\end{proof}

\begin{lemma}\label{Fix-Gamma-j}
Let $M \in GL_n$ and assume that $\Gamma_M \circ j$ is elliptic.  Then

$Fix(\Gamma_M \circ j) \subseteq Fix(\Gamma_{MM^{-T}})$.
\end{lemma}

\begin{proof}
If $P \in Fix(\Gamma_M \circ j)$, then $MP^{-1}M^T=P$, i.e. $M^{-T}PM^{-1} = P^{-1}$. Hence we have: $MM^{-T} P (MM^{-T})^T = M(M^{-T} P M^{-1}) M^T = MP^{-1}M^T = P$ and so $P \in Fix(\Gamma_{MM^{-T}})$.
\end{proof}

\begin{rems}\label{ort-norm}
a) Let $M \in GL_n$, then $M$ is normal if and only if $MM^{-T}$ is orthogonal.

b)  Let $M, S \in GL_n$. Then $S^{-1}M S^{-T}$ is normal if and only if $S^{-1} MM^{-T}S$ is orthogonal.

Indeed $M^T M = M M^T$ if and only if $(MM^{-T})^T(MM^{-T}) = I_n$ and this gives (a).

Part (b) follows from (a), since $S^{-1} M M^{-T} S = (S^{-1} M S^{-T})(S^{-1} M S^{-T})^{-T}$.
\end{rems}

\begin{remsdefs}\label{S-Q}
Let $M \in GL_n$, assume that $\Gamma_M \circ j$ is elliptic or, equivalently, that $M$ is $\mathbb{R}$-congruent to an orthogonal matrix.

a) By Proposition \ref{class-ellitt-Gamma-j}, we denote by $S$ a matrix in $GL_n$ such that $S^{-1} M M^{-T} S \in \mathcal{O}_n$. 
By Remarks \ref{ort-norm} (b), $S^{-1} M S^{-T}$ is normal and so, by Theorem \ref{polar-dec} and Remark \ref{dec-pol-norm}, $S^{-1} M S^{-T} = QU = UQ = \sqrt{Q}\, U \sqrt{Q}$ where $Q \in \mathcal{P}_n$ and $U \in \mathcal{O}_n$ are the components of the polar decomposition of $S^{-1} M S^{-T}$. This gives also that $M$ is $\mathbb{R}$-congruent to $U$.

Conversely, it is easy to verify that every $U \in \mathcal{O}_n$, which is $\mathbb{R}$-congruent to $M$, can be obtained in this way, starting from a matrix $S \in GL_n$ such that  $S^{-1} M M^{-T} S$ is orthogonal.

Fixed a matrix $S$ as above, the explicit expression for $U$ is the following:

$U= (\sqrt{S^{-1}M S^{-T} S^{-1}M^T S^{-T}}\,)^{-1} S^{-1}M S^{-T}$.

In particular, if $M$ is real normal, then we can choose 
$S=I_n$, so that $U= (\sqrt{MM^T})^{-1}M$. 

b) Following Remark-Definition \ref{J-tilda}, we denote by $Z \in \mathcal{O}_n$ a matrix such that $U= Z \widetilde{J}_U Z^T$ and, so, for $R= S \sqrt{Q}Z$ we get: $M= R \widetilde{J}_U R^T$.
\end{remsdefs}

\begin{prop}\label{Fix-Gamma-j-matrici}
Let $M \in GL_n$, assume that $\Gamma_M \circ j$ is elliptic, and let $U$ be an orthogonal matrix, $\mathbb{R}$-congruent to $M$, and $R \in GL_n$ such that $M= R \widetilde{J}_U R^T$
(remember Remarks-Definitions \ref{S-Q}).
 Then
\begin{center}
$Fix(\Gamma_M \circ j) = \Gamma_R(\{ GG^T  : G \in \mathcal{C}_{J_{MM^{-T}}}, GG^T \in \mathcal{K}_ {\widetilde{J}_U}\})
$.
\end{center}
\end{prop}

\begin{proof}
By Lemma \ref{Fix-Gamma-j}, $Fix(\Gamma_M \circ j) \subseteq Fix(\Gamma_{MM^{-T}})$.

Now $MM^{-T} = R \widetilde{J}_U R^T R^{-T} \widetilde{J}_U R^{-1} = R (\widetilde{J}_U)^2 R^{-1} = R J_{U^2} R^{-1} = R J_{MM^{-T}} R^{-1}$. By Proposition \ref{Fix-Gamma} (ii), if $P \in Fix(\Gamma_M \circ j)$, then there exists $G \in \mathcal{C}_{J_{MM^{-T}}}$ such that $RGG^TR^T=P$. Moreover $MP^{-1}M^T = P$ is equivalent to $\widetilde{J}_U = (GG^T) \widetilde{J}_U (GG^T)^T$, i.e. to $GG^T \in \mathcal{K}_ {\widetilde{J}_U}$; hence $Fix(\Gamma_M \circ j) \subseteq \{ R GG^T R^T : G \in \mathcal{C}_{J_{MM^{-T}}}, GG^T \in \mathcal{K}_ {\widetilde{J}_U}\}= \Gamma_R(\{ GG^T  : G \in \mathcal{C}_{J_{MM^{-T}}}, GG^T \in \mathcal{K}_ {\widetilde{J}_U}\})$.

For the other inclusion, let $P= RGG^TR^T$ with $G \in \mathcal{C}_{J_{MM^{-T}}}$ and $GG^T \in \mathcal{K}_ {\widetilde{J}_U}$. Then we have:
$M P^{-1} M^T = R \widetilde{J}_U (GG^T)^{-1} (\widetilde{J}_U)^T R^T = R GG^T \widetilde{J}_U (\widetilde{J}_U)^T R^T =P$ and then $P \in Fix(\Gamma_M \circ j)$ and the Proposition is proved.
\end{proof}

\begin{example}\label{Omega}
Now let $M= \Omega_p:= I_p\oplus (-I_{n-p})$ with $p = 0, \cdots , n$. 

Note that $\Omega_p$ is diagonal, orthogonal, $\Omega_p^2 = I_n$ and that $\Omega_p$ agrees with its RJA form $\widetilde{J}_{\Omega_p}$. Hence, by Proposition \ref{class-ellitt-Gamma-j}, $\Gamma_{\Omega_p} \circ j$ is elliptic  and with the same notations of Remarks-Definitions \ref{S-Q}, we can choose $S= Q = Z = I_n$ and $U = \Omega_p$.

Hence, by Proposition \ref{Fix-Gamma-j-matrici}, $Fix(\Gamma_{\Omega_p} \circ j) = \mathcal{P}_n \cap \mathcal{O}(p, n-p) = \mathcal{P}_n \cap S\mathcal{O}_0(p, n-p) $.

Moreover, $S\mathcal{O}_0(p, n-p)$ is a reductive subgroup of $GL_n$, satisfying the condition (*) in Proposition \ref{Gruppi-alg-riduttivi} and we have also $S\mathcal{O}_0(p, n-p) \cap S\mathcal{O}_n = S\mathcal{O}_p \oplus S\mathcal{O}_{n-p}$ since the inclusion $S\mathcal{O}_0(p, n-p) \cap S\mathcal{O}_n \supseteq S\mathcal{O}_p \oplus S\mathcal{O}_{n-p}$ is trivial, while it is well-known that $S\mathcal{O}_p \oplus S\mathcal{O}_{n-p}$ is a maximal compact subgroup of $S\mathcal{O}_0(p, n-p)$ (see for instance \cite[Ch.\,VI]{Helg2001}). By Proposition \ref{Gruppi-alg-riduttivi}, this allows to conclude that $Fix(\Gamma_{\Omega_p} \circ j)= \mathcal{P}_n \cap S\mathcal{O}_0(p, n-p)$ is diffeomorphic to the irreducible symmetric space $S\mathcal{O}_0(p, n-p)/\big( S\mathcal{O}_p \oplus S\mathcal{O}_{n-p}\big)$, whose dimension is $p(n-p)$.
\end{example}

\begin{example}\label{Lambda}
Consider the case: $n= 2m$ and $M = \Lambda_m := E^{\oplus m}$. 

Arguing  as in Example \ref{Omega}, we get that $\Lambda_m$ is orthogonal, skew-symmetric, $\Lambda_m^2 =-I_{2m}$ and $\Lambda_m$ agrees with its RJA form $\widetilde{J}_{\Lambda_m}$: $\Lambda_m= \widetilde{J}_{\Lambda_m}$. Hence  $\Gamma_{\Lambda_m} \circ j$ is elliptic and we can get: $S=Q=Z=I_{2m}$ and $U = \Lambda_m$.

Therefore $Fix(\Gamma_{\Lambda_m} \circ j) =\{GG^T : GG^T \in \mathcal{K}_{\Lambda_m} \}$.

Now let $W$ be an orthogonal matrix such that
$\Lambda_m= W
\begin{pmatrix} 
 0 & I_m \\ 
-I_m &  0
\end{pmatrix}
W^T
$. 

It is easy to verify that $\mathcal{K}_{\Lambda_m} = \Gamma_W(Sp_{2m})$. Then we get

$Fix(\Gamma_{\Lambda_m} \circ j) = \mathcal{P}_{2m} \cap \mathcal{K}_{\Lambda_m} = \mathcal{P}_{2m} \cap \Gamma_W(Sp_{2m}) = \Gamma_W(\mathcal{P}_{2m} \cap Sp_{2m})$.

Moreover, since $Sp_{2m}$ is an algebraic reductive subgroup of $GL_{2m}$ and $Sp_{2m} \cap \mathcal{O}_{2m}=\rho(U_{m})$ (see for instance \cite[Prop. 2.12, p. 33]{deGoss2006}), by Proposition \ref{Gruppi-alg-riduttivi} we get that $Fix(\Gamma_{\Lambda_m} \circ j)= \Gamma_W(\mathcal{P}_{2m} \cap Sp_{2m})$ is diffeomorphic to the irreducible symmetric space $Sp_{2m}/\rho(U_{m})$, whose dimension is $m(m+1)$.
\end{example}

\begin{example}\label{Theta}
Now let $\theta \in (0, \dfrac{\pi}{2})$, $\mu \ge 0$, $\nu \ge 0$, $\mu  +\nu \ge 1$ and $n= 2(\mu  +\nu)$.
Let $M= \Theta_{\theta; \mu, \nu} := E_\theta^{\oplus \mu} \oplus (-E_\theta^{\oplus \nu})$.

Once again arguing as in Example \ref{Omega}, we get that $\Theta_{\theta; \mu, \nu}$ is orthogonal, skew-symmetric, $\Theta_{\theta; \mu, \nu}^2 = E_{2 \theta}^{\mu+\nu}$ and $\Theta_{\theta; \mu, \nu}$ agrees with its RJA form  $\widetilde{J}_{\Theta_{\theta; \mu, \nu}}$. Hence $\Gamma_{\Theta_{\theta; \mu, \nu}} \circ j$ is elliptic and we can choose $S=Q=Z=I_{2(\mu+\nu)}$, $U = \Theta_{\theta; \mu, \nu}$ so that

$Fix(\Gamma_{\Theta_{\theta; \mu, \nu}} \circ j) = \{ GG^T : G \in \mathcal{C}_{E_{2 \theta}^{\mu+\nu}}, GG^T \in \mathcal{K}_{\Theta_{\theta; \mu, \nu}} \}$.

By Lemma \ref{lemma-su-J}, $\mathcal{C}_{E_{2 \theta}^{\mu+\nu}}= \rho(GL_n(\mathbb{C}))$, where $\rho: GL_{\mu+\nu}(\mathbb{C}) \to GL_{2(\mu+\nu)}$ is the monomorhism defined in Remark \ref{rho}.
Note that $\Theta_{\theta; \mu, \nu}= \rho(e^{i\theta}(I_\mu \oplus (-I_\nu)))$. 

Since $\rho$ preserves products and it is injective, we get that 
$Fix(\Gamma_{\Theta_{\theta; \mu, \nu}} \circ j) =\\
=\rho\big(\{HH^* : H\in GL_{\mu+\nu}(\mathbb{C}), HH^* (e^{i \theta}(I_\mu \oplus (-I_\nu))) HH^* = 
e^{i \theta}(I_\mu \oplus (-I_\nu))\} \big)=\\
\rho \big(\{HH^* : H\in GL_{\mu+\nu}(\mathbb{C}), HH^* (I_\mu \oplus (-I_\nu))) HH^* = 
(I_\mu \oplus (-I_\nu))\} \big)=\\
=\rho(\mathcal{H}_{\mu+\nu} \cap U(\mu, \nu)) = \rho(\mathcal{H}_{\mu+\nu} \cap SU(\mu, \nu))$.

Now $\rho(\mathcal{H}_{\mu+\nu} \cap SU(\mu, \nu)) = \mathcal{P}_{2(\mu+\nu)} \cap \rho(SU(\mu, \nu))$ is an algebraic reductive subgroup of $GL_n$ and, arguing as in Example \ref{Omega}, $\rho(SU(\mu, \nu)) \cap\mathcal{O}_{2 (\mu+\nu)} = \rho(S(U_\mu \oplus U_\nu))$ (where $S(U_\mu \oplus U_\nu)$ consists in matrices of $U_\mu \oplus U_\nu$ with determinant $1$); indeed $\rho(S(U_\mu \oplus U_\nu))$ is trivially included in $\rho(SU(\mu, \nu)) \cap\mathcal{O}_{2 (\mu+\nu)}$ and $S(U_\mu \oplus U_\nu)$ is a maximal compact subgroup of $SU(\mu, \nu)$.
By Proposition \ref{Gruppi-alg-riduttivi}, we obtain that $Fix(\Gamma_{\Theta_{\theta; \mu, \nu}} \circ j)= \mathcal{P}_{2(\mu+\nu)} \cap \rho(SU(\mu, \nu)) = \rho(\mathcal{H}_{\mu+\nu} \cap SU(\mu, \nu))$ is diffeomorphic to the irreducible symmetric space  $SU(\mu, \nu) / S(U_\mu \oplus U_\nu)$, whose dimension is $2\mu\nu$.
\end{example}

\begin{prop}\label{Gamma-j-Riem-prod}
Let $M \in GL_n$, assume that $\Gamma_M \circ j$ is elliptic, and let $U$ be an orthogonal matrix, $\mathbb{R}$-congruent to $M$
(remember Remarks-Definitions \ref{S-Q}). 

Let $
\widetilde{J}_U = I_p \oplus (- I_q) \oplus E_{\phi_1}^{\oplus \mu_1} \oplus (- E_{\phi_1 }^{\oplus \nu_1}) \oplus \cdots \oplus E_{\phi_h}^{\oplus \mu_h} \oplus (- E_{\phi_h}^{\oplus \nu_h}) \oplus E^{\oplus k}
$ the RJA form of $U$. 
Then $Fix(\Gamma_M \circ j)$ and

$(\mathcal{P}_{p+q} \cap S\mathcal{O}_0(p,q)) \oplus 
\rho(\mathcal{H}_{\mu_1 + \nu_1} \cap U(\mu_1, \nu_1))\oplus  \cdots \oplus \rho(\mathcal{H}_{\mu_h + \nu_h} \cap U(\mu_h, \nu_h))
\oplus (\mathcal{P}_{2k} \cap Sp_{2k})$ are  isometric (by congruence) as Riemannian submanifolds of $(\mathcal{P}_n, g)$.

In particular $(Fix(\Gamma_M \circ j), g)$ is a closed simply connected totally geodesic symmetric Riemannian submanifold of $(\mathcal{P}_n, g)$ of dimension $pq + 2\sum_{j=1}^h \mu_j\nu_j + k(k+1)$, isometric to the Riemannian product
\begin{center}
$(\mathcal{P}_{p+q} \cap S\mathcal{O}_0(p,q), g) \times \, \prod_{j=1}^h (\mathcal{H}_{\mu_j + \nu_j} \cap U(\mu_j, \nu_j), 2\gamma) \, \times (\mathcal{P}_{2k} \cap Sp_{2k}, g)$.
\end{center}
\end{prop}

\begin{proof}
We have: $J_{MM^{-T}} = J_{U^2} = (\widetilde{J}_U)^2 =\\
I_{p+q} \oplus E_{2 \phi_1}^{\oplus(\mu_1+\nu_1)} \oplus \cdots \oplus E_{2 \phi_h}^{\oplus(\mu_h+\nu_h)} \oplus (-I_{2k})$, 
hence, by Lemma \ref{lemma-su-J}, 

$\mathcal{C}_{J_{MM^{-T}}} = GL_{p+q} \oplus \rho(GL_{\mu_1+\nu_1}(\mathbb{C})) \oplus 	\cdots \oplus \rho(GL_{\mu_h+\nu_h}(\mathbb{C})) \oplus GL_{2k}$.

Therefore: $\{ GG^T : G\in \mathcal{C}_{J_{MM^{-T}}} \}= \mathcal{P}_{p+q} \oplus \rho(\mathcal{H}_{\mu_1+\nu_1}) \oplus \cdots \oplus \rho(\mathcal{H}_{\mu_h+\nu_h}) \oplus \mathcal{P}_{2k}$.

Now, arguing as in Examples \ref{Omega},
 \ref{Lambda} and  \ref{Theta},
we get 

$\{ GG^T : G\in \mathcal{C}_{J_{MM^{-T}}} \} \cap \mathcal{K}_{\widetilde{J}_U} = \\
(\mathcal{P}_{p+q} \cap S\mathcal{O}_0(p,q)) \oplus 
\rho(\mathcal{H}_{\mu_1 + \nu_1} \cap U(\mu_1, \nu_1))\oplus  \cdots \oplus \rho(\mathcal{H}_{\mu_h + \nu_h} \cap U(\mu_h, \nu_h)) \oplus \\
\oplus \Gamma_W (\mathcal{P}_{2k} \cap Sp_{2k})
$. 

We conclude by Proposition \ref{Fix-Gamma-j-matrici}.
\end{proof}

\begin{prop}\label{DeRham-Gamma-j}
Let $M \in GL_n$, assume that $\Gamma_M \circ j$ is elliptic and let $U$ be an orthogonal matrix, $\mathbb{R}$-congruent to $M$, having $1$ as eigenvalue with multiplicity $ p \ge0$, $-1$ as eigenvalue with multiplicity $ q \ge0$, $i$ as eigenvalue with multiplicity $k \ge0$ and call $\phi_1, \cdots , \phi_h$ ($h \ge 0$), the set of mutually distinct possible values in $(0, \dfrac{\pi}{2})$ such that $e^{i \phi}$ or $e^{i(\pi-\phi)}$ is an eigenvalue of $U$.

Moreover, if $h > 0$, for every $j = 1, \cdots, h$, we set to be $\mu_j$ the multiplicity of $e^{i \phi_j}$ and $\nu_j$  to be the multiplicity of $e^{i (\pi -\phi_j)}$ (note that $\mu_j \ge 0$, $\nu_j \ge 0$, $\mu_j + \nu_j \ge 1$ and $p, q, k, h$ are not all zero).

Then, up to isometries,  the De Rham factors of $\left( Fix(\Gamma_M \circ j), g \right)$ are:

a) $\mathbb{R}$, if $p=q=1$;

b) $ S\mathcal{O}_0(p, q)/ ( S\mathcal{O}_p \oplus S\mathcal{O}_q )$, if $p, q \ge 1$ and $p+q \ge 3$;

c) $Sp_{2k}/\rho(U_{k})$, if $k \ge 1$;

d) $SU(\mu_j, \nu_j) / S(U_{\mu_j}  \oplus  U_{\nu_j})$, if $h \ge 1$ for every $j = 1, \cdots , h$ such that $\mu_j, \nu_j \ge 1$.
\end{prop}

\begin{proof}
It follows  from Proposition \ref{Gamma-j-Riem-prod} and from Examples \ref{Omega}, \ref{Lambda} and \ref{Theta}, by analogous arguments, developed in Remark \ref{DeRham-casi}.
\end{proof}

\section{The fixed points of the isometries $\Gamma_M \circ j \circ \delta$}

\begin{prop}\label{class-ellitt-Gamma-j-chi}
Let $M \in GL_n$ and let us consider the isometry $\Gamma_M \circ j \circ \delta$ of $\mathcal{P}_n$.
The following facts are equivalent:

a) $\Gamma_M \circ j \circ \delta$ is elliptic;

b) $\Gamma_M \circ j$ is elliptic and $det(M)= \pm 1$;

c) $M$ and $M^{-T}$ are $\mathbb{R}$-congruent via a positive definite matrix and $det(M)= \pm 1$;

d) $M$ is $\mathbb{R}$-congruent to an orthogonal matrix and $det(M)= \pm 1$; 

e) $M$ is $\mathbb{R}$-congruent to a normal matrix and $det(M)= \pm 1$; 

f) $\Gamma_{M M^{-T}}$ is elliptic and $det(M)= \pm 1$;

g) $MM^{-T}$ is semi-simple with eigenvalues of modulus $1$ and $det(M)= \pm 1$.
\end{prop}

\begin{proof}
It suffices to prove the equivalence (a) $\Leftrightarrow$ (d) and the other equivalences will follow from Proposition \ref{class-ellitt-Gamma-j}.

Assume first that $det(P)^{2/n} M P^{-1} M^T =  P$, then we get $det(M)= \pm 1$ simply by computing the determinants. After setting $P=CC^T$, we get: 

$det(CC^T)^{2/n} M (CC^T)^{-1} M^T =  CC^T$, and thus 

$([\dfrac{C}{|det(C)|^{1/n}}]^{-1} M [\dfrac{C}{|det(C)|^{1/n}}]^{-T}])([\dfrac{C}{|det(C)|^{1/n}}]^{-1} M [\dfrac{C}{|det(C)|^{1/n}}]^{-T}])^T = I_n$,

i.e. 
$[\dfrac{C}{|det(C)|^{1/n}}]^{-1} M [\dfrac{C}{|det(C)|^{1/n}}]^{-T} \in \mathcal{O}_n$ and we get (a) $\Rightarrow$ (d).

For the converse, assume that $M=KUK^T$ with $K \in GL_n$, $U \in \mathcal{O}_n$ and $det(M) = \pm 1$. By computing the determinants we obtain: $det(U) = det (M)$ and $det(KK^T)=1$. Hence: $det(KK^T)^{2/n} M (KK^T)^{-1} M^T = M K^{-T}K^{-1} M^T = KK^T$ (after replacing $M$ with $KUK^T$), i.e. $KK^T$ is a fixed point of $\Gamma_M \circ j \circ \delta$.
\end{proof}

\begin{prop}\label{Fix-Gamma-j-chi}
Let $M \in GL_n$ and assume that $\Gamma_M \circ j \circ \delta$ is elliptic.  Then

a) $\emptyset \ne Fix(\Gamma_M \circ j) \subseteq SL\mathcal{P}_n$;

b) $Fix(\Gamma_M \circ j \circ \delta) = \mathbb{R}^+ \cdot \, Fix(\Gamma_M \circ j):= \{tP : t > 0, P \in Fix(\Gamma_M \circ j)\}$;

c) $(Fix(\Gamma_M \circ j \circ \delta), g)$ a closed simply connected totally geodesic symmetric Riemannian submanifold of $(\mathcal{P}_n, g)$, isometric to the Riemannian product  

$(Fix(\Gamma_M \circ j), g) \times (\mathbb{R}, \varepsilon)$ ($\varepsilon$ is the ordinary euclidean metric) and hence, with the same notations as in Proposition \ref{Gamma-j-Riem-prod}, it is isometric to the Riemannian product

$(\mathcal{P}_{p+q} \cap S\mathcal{O}_0(p,q), g) \times \, \prod_{j=1}^h (\mathcal{H}_{\mu_j + \nu_j} \cap U(\mu_j, \nu_j), 2\gamma) \, \times (\mathcal{P}_{2k} \cap Sp_{2k}, g) \times (\mathbb{R}, \varepsilon)$ 

and its dimension is $pq + 2\sum_{j=1}^h \mu_j \nu_j + k(k+1) +1$.
\end{prop}

\begin{proof}

a) It follows from Proposition \ref{class-ellitt-Gamma-j-chi} and from Remark \ref{det-P-j-chi}.

b) If $P \in Fix(\Gamma_M \circ j)$, then $MP^{-1}M^T= P $ and $det(P )= 1$ from (a).

Hence, for every $t \in \mathbb{R}^+$,
$det(tP)^{2/n} M(tP)^{-1}M^T = tMP^{-1}M^T = tP$, 

i.e. $tP \in Fix(\Gamma_M \circ j \circ \delta)$.

For the other inclusion, it suffices to note that, if $P \in Fix(\Gamma_M \circ j \circ \delta)$, then $\dfrac{P}{det(P)^{1/n}} \in Fix(\Gamma_M \circ j)$.

Indeed we have: $M(\dfrac{P}{det(P)^{1/n}})^{-1} M^T = det(P)^{1/n} MP^{-1}M^T = \dfrac{P}{det(P)^{1/n}}$.

c) The mapping $(\mathcal{P}_n,g) \to (SL\mathcal{P}_n, g) \times (\mathbb{R}, \varepsilon)$, $P \mapsto (\dfrac{P}{det(P)^{1/n}}, \dfrac{\ln(det(P))}{\sqrt{n}})$, is an isometry, as proved in \cite[Proof of Prop.\,2.7]{DoPe2019}. By part (b), the restriction of this mapping to $Fix(\Gamma_M \circ j \circ \delta)$ is an isometry from $(Fix(\Gamma_M \circ j \circ \delta), g)$ onto $(Fix(\Gamma_M \circ j), g) \times (\mathbb{R}, \varepsilon)$. 
\end{proof}

\begin{prop}\label{DeRham-Gamma-j-delta}
Let $M \in GL_n$, assume that $\Gamma_M \circ j \circ \delta$ is elliptic, let $U$ an orthogonal matrix, $\mathbb{R}$-congruent to $M$, and let $p, q, k, h, \mu_j, \nu_j  \ge 0$ as in Proposition \ref{DeRham-Gamma-j}.

Then, up to isometries,  the De Rham factors of $\left( Fix(\Gamma_M \circ j \circ \delta), g \right)$ are:

a) $\mathbb{R}^2$, if $p=q=1$;

b) $\mathbb{R}$, if $p \ne 1$ or $q \ne 1$;

c) $ S\mathcal{O}_0(p, q)/ ( S\mathcal{O}_p \oplus S\mathcal{O}_q )$, if $p, q \ge 1$ and $p+q \ge 3$;

d) $Sp_{2k}/\rho(U_{k})$, if $k \ge 1$;

e) $SU(\mu_j, \nu_j) / S(U_{\mu_j}  \oplus  U_{\nu_j})$, if $h \ge 1$ for every $j = 1, \cdots , h$ such that $\mu_j, \nu_j \ge 1$.
\end{prop}

\begin{proof}
It follows directly from Propositions \ref{Fix-Gamma-j-chi} and \ref{DeRham-Gamma-j}.
\end{proof}

\bibliographystyle{alpha}
\bibliography{DolPerArx}

\end{document}